\documentclass[12pt]{amsart}
\usepackage{amsmath,amsfonts,amscd,bezier,amsthm,amssymb}
\usepackage{mathrsfs,mathtools}
\usepackage{cite}
\usepackage{graphicx}
\usepackage{float}
\usepackage{enumitem,xfrac,nccmath}
\usepackage{tasks, multicol,yhmath}
\usepackage{hyperref}
\usepackage{amsrefs}
\usepackage{cleveref}
\usepackage{tikz}
\usetikzlibrary{graphs}
\usepackage{comment}

\usepackage[pagewise]{lineno}

\theoremstyle{plain}
\newtheorem{thm}{Theorem}
\newtheorem{prop}[thm]{Proposition}

\theoremstyle{remark}
\newtheorem*{rmk}{Remark}
\newtheorem*{rmks}{Remarks}
\theoremstyle{definition}

\newcommand{\N}{\mathbb{{N}}}

\newcommand{\C}{\mathbb{{C}}}

\newcommand{\eps}{{\varepsilon}}

\newcommand{\Col}{\mathcal{{C}}}
\renewcommand{\H}{\mathcal{{H}}}
\newcommand{\GG}{\mathcal{{G}}}
\newcommand{\TT}{\mathcal{{T}}}

\DeclareMathOperator{\hijos}{Chi}
\DeclareMathOperator{\des}{Des}

\numberwithin{equation}{section}

\allowdisplaybreaks

\makeatletter
\DeclareRobustCommand{\eqrefp}[2]{%
	\textup{\tagform@{\refp{#1}{#2}}}}
\makeatother
\DeclareRobustCommand{\refp}[2]{%
	\expandafter\ifx\csname r@#1\endcsname\relax
	\textbf{??}%
	\else
	\edef\areferencia{\ref{#1}-)}%
	\expandafter\eqrefpaux\areferencia-#2%
	\fi}
\def\eqrefpaux#1-#2){#1}

\begin{document}
	\title[Shift universality]{Rota universality of shifts on trees}

        \author[J. Buckley]{Jeremiah Buckley}
	\address[J. Buckley]{Department of Mathematics, Faculty of Natural, Mathematical and Engineering Sciences, S4.02.1, Strand Building, Strand Campus, Strand, London WC2R 2LS.}
	\email{jeremiah.buckley@kcl.ac.uk}

	\author[D. Seco]{Daniel Seco}
	\address[D. Seco]{Departamento de An\'alisis Matem\'atico e IMAULL,  Universidad de La Laguna, Avenida Astrof\'isico Francisco S\'anchez, s/n, 38206 San Crist\'obal de La Laguna, Santa Cruz de Tenerife,  Spain} \email{dsecofor@ull.edu.es}

	\thanks{The first author was partially supported by EPSRC: NIA EP/V002449/1. The second author was supported by Grant PID2024-160185NB-I00 funded by MICIU/AEI/10.13039/501100011033 and by FEDER, UE, through the Generaci\'on de Conocimiento programme of Spanish Ministry of Science, Innovation and Universities;  and by grant RYC2021-034744-I of the Ram\'on y Cajal programme from Agencia Estatal de Investigaci\'on (Spanish Ministry of Science, Innovation and Universities).}

	

	\date{\today}
	
	\begin{abstract} 
	We consider shift operators on infinite directed trees. We prove that under very mild assumptions on the structure of the tree, the resulting operator is universal (in the sense of Rota). We show connections between the spectrum of the operator and the frequency of branching in the tree.
	\end{abstract}
	
	\maketitle
	
\section{Introduction}

A recurring theme in operator theory is that hard analytic questions can be recast in terms of a single, well-chosen operator. Rota universality is a prototypical example of this idea. An operator $U$ on a separable infinite-dimensional Hilbert space $\H$ is \emph{universal} if for every non-zero bounded operator $T$ on $\H$ there are an invariant subspace $M$ of $U$, a scalar $\omega\in\C$ and an isomorphism $A:\H\to M$ with $UA=\omega AT$. In other words, every operator is similar to a scalar multiple of the restriction of $U$ to an invariant subspace. In particular, the Hilbert invariant subspace problem (HISP) is equivalent to understanding the invariant subspaces of any one universal operator.  A good reference for Rota universality is \cite{CowenGallardo} while for the HISP we recommend \cite{ChalPart}. Known examples of universal operators come from classical topics in analysis: adjoints of shifts on function spaces, certain Toeplitz and composition operators. In this paper we look at a different source, \emph{shifts on trees}.  We will show that, under some mild assumptions, the shift operator $S$ on a directed tree is universal in the sense of Rota, and so is $S-\lambda I$ for an explicit set of values of $\lambda$. This will involve studying the spectrum of $S$. 

Let $\TT$ be a (directed) tree, let $V$ be its vertices and let $\Col$ denote the map that sends a vertex to its parent. We are interested in the shift map $S$ on the Hilbert space $\ell^2(V)$ given by
\begin{equation*}
    S e_v= 
    \begin{cases}
    0\text{ if } v \text{ is a root},\\
    e_{\Col(v)}\text{ otherwise.}
    \end{cases}
\end{equation*}  
We refer to \cite{Jabetal} for the definition of a root, and for a comprehensive study of classical operator theory on directed trees. We will assume that there is a maximum bound on the number of children of each vertex (see condition (A3) below). In that case, the operator $S$ can be extended to a bounded operator on all of $\ell^2(V)$ by continuity and linearity. 

By descendants of a vertex $v$ we mean all elements $u$ of $V$ such that there exists $n \in \N$ for which $\Col^n(u)=v$. When $n=1$ we refer to \emph{children} of $v$, the set of which is denoted by $\hijos v$.

Throughout we shall assume that $\TT$ satisfies the following:
\begin{enumerate}[label={(A\arabic*)}]
    \item $V$ is countable.\label{Ass: count}
    \item $\TT$ is leafless, that is, $\hijos v\neq\varnothing$ for every $v\in V$.\label{Ass: leafless}
    \item\label{Ass: bdd branch} There exists $B>0$ such that $\#\hijos v\leq B$ for every $v\in V$, that is, no vertex has more than $B$ children.
    \item\label{Ass: lots of branch} There exists an infinite set of vertices $v\in V$ whose descendants are pairwise disjoint and satisfying $\#\hijos v\geq 2$.
\end{enumerate}

\begin{rmks}
    Let us make some comments on our assumptions.
    \begin{enumerate}[label=(\roman*)]
        \item Assumption \ref{Ass: count} is necessary for $S$ and its adjoint $S^*$ to have a reasonable definition \cite{Jabetal}*{Proposition 3.1.10}. In any case, since trees are by definition connected, $V$ is countable for any tree satisfying (A3). 
            \item Assumption \ref{Ass: leafless} is simply convenient, and our results can be applied to trees with leaves by ``pruning'', i.e., by considering the subtree where finite branches are eliminated and applying our techniques to the subtree. Similarly, by pruning the parts of the tree with little or no branching, one may improve the results.

        \item Assumption \ref{Ass: bdd branch}  limits the frequency of branching present in the tree which guarantees that $S$ is a well-defined bounded operator.
        \item Assumption \ref{Ass: lots of branch} ensures that there is enough branching in $\TT$ to make things interesting.
        \end{enumerate}     
\end{rmks}

We have two main examples in mind when we make these assumptions, directed binary trees and the Collatz graph (although strictly speaking one needs to eliminate cycles from the graph). In fact, our original motivation was to study the dynamics of the Collatz map from an operator-theoretic perspective. This connects tangentially with the treatment in \cite{Mori}. We give more details on these special cases in \Cref{sec: exas}.

To state our main results, we define
\begin{equation*}
    \kappa_m(v)=\#\{u\in V:\Col^{(m)}(u)=v\}\leq B^m
\end{equation*}
to be the number of descendants of $v\in V$ that are exactly $m$ generations below it, and put
\begin{equation*}
    \kappa_m^{(+)}=\sup_{v\in V} \kappa_m(v)\qquad\text{and}\qquad\kappa_m^{(-)}=\inf_{v\in V} \kappa_m(v).
\end{equation*}
It is easy to see that $\kappa_{m+\ell}^{(+)}\le \kappa_m^{(+)}\kappa_\ell^{(+)}$ and $\kappa_{m+\ell}^{(-)}\geq\kappa_m^{(-)}\kappa_\ell^{(-)}$ since, e.g., if we wish to compute $\kappa_{m+\ell}(v)$ then $v$ has at most $\kappa_m^{(+)}$ descendants exactly $m$ generations below it, and each of these has at most $\kappa_\ell^{(+)}$ descendants exactly $\ell$ generations below them. This means that $\log \kappa_m^{(+)}$ (respectively $\log \kappa_m^{(-)}$) is a sub-additive (respectively super-additive) sequence and so
\begin{equation*}
    t_{\pm} = \lim_{m\to\infty} \frac1m \log \kappa_m^{(\pm)}
\end{equation*}
exist, by Fekete's lemma. We put $R_{\pm}=\exp(\frac12 t_{\pm})$.

The following is our main result, which is deliberately simple to prove yet powerful:

\begin{thm}\label{thm1} 
    \begin{enumerate}[label=(\alph*)]
        \item\label{thm: part univ} For all $|\lambda|< R_-$, $\lambda$ is an eigenvalue of $S$. Moreover, $S- \lambda I$ is universal.
        \item\label{thm: part spec rad} The spectral radius of $S$ is $R_+$.
    \end{enumerate}
\end{thm}

\begin{rmks}
    \phantom{w}
    \begin{enumerate}[label=(\roman*)]
        \item Our proof that $S$ is universal involves showing that it satisfies Caradus' condition \cite{Car}. A refinement of this condition due to Cowen and Gallardo-Guti\'{e}rrez \cite{CG1}*{Theorem 2} shows that in this case there exists $\eps>0$ such that $S- \lambda I$ is universal for all $|\lambda|< \eps$. Our theorem gives a quantifiable version. Similar results for composition operators are given in \cite{Mat}, though we are not aware of any such results for shift operators.
        \item Since for the directed binary tree $R_-=R_+$, both bounds are sharp. We will give a full spectral description of this example in Section~\ref{sec: exas}.
        \item The existence of the limit defining $R_+$ follows directly from the proof of Part~\ref{thm: part spec rad} of the theorem via Gelfand's formula, without needing to appeal to Fekete's Lemma.
        \item By \cite{Jabetal}*{Corollary 3.3.2} the spectrum of $S$ is circularly symmetric, in particular it contains the circle of radius $R_+$. A natural question is to understand the intersection of the spectrum with the annulus $R_-<|\lambda|<R_+$. 
        
    \end{enumerate}     
\end{rmks}

We identify two natural directions that our work points to:
Can we use operator theory to study properties of trees? For instance, in Section~\ref{sec: Collatz exa} we show how our results reformulate an old conjecture on the growth of the Collatz tree. And can the properties of the tree shed any light on the invariant subspace structure of the operator? In particular, is the Hilbert invariant subspace problem tractable from a graph theoretic point of view?

Our plan is as follows: In \Cref{sec: exas} we will make some comments on the two mentioned examples that we have in mind. Then, in \Cref{sec: 3} we will discuss simple operator theoretic properties of $S$, and compute a formula for $S^*$. We will then use this to prove part (b) of our main Theorem. We conclude in \Cref{sec: univ} by establishing part (a).

\section{Examples}\label{sec: exas}
As we alluded to in the introduction, we are mainly motivated by two examples, that we consider enlightening.

\subsection{Directed binary tree}
We say that a directed tree $\mathcal B$ satisfying Assumptions \labelcref{Ass: count,Ass: bdd branch,Ass: leafless,Ass: lots of branch} is a \emph{directed binary tree} if additionally every vertex $v\in V$ has exactly two children, i.e., $\#\hijos(v)=2$.

 There are essentially only two binary trees (with or without a root). In the case with a root, the graph can be interpreted as modelling the two left shifts $S_1$ and $S_2$ over the non-commutative bidisk, but we take the opposite perspective: our shifts are backwards in such a representation and act by eliminating from each monomial the left-most variable $z_j$.

Clearly in this case we have $R_+=R_-=\sqrt{2}$. Let us first show that our result is sharp, that is, if $|\lambda|=\sqrt{2}$ then $\lambda$ is not an eigenvalue. Indeed, if $\lambda\neq 0$ is an eigenvalue and
\begin{equation*}
    x=\sum_{v\in V}a_v e_v\neq 0
\end{equation*}
a corresponding eigenvector then
\begin{equation*}
    \sum_{u\in\hijos(v)}a_u=\lambda a_v
\end{equation*}
for every $v\in V$. Choose $v_0$ such that $a_{v_0}\neq 0$ and, applying the elementary inequality $|z|^2+|w|^2\geq \frac12|z+w|^2$, notice that 
\begin{equation*}
    \sum_{u\in\hijos(v_0)}|a_u|^2\ge \frac{|\lambda|^2}{2} |a_{v_0}|^2.
\end{equation*}
From now on we denote by $\hijos^{\langle n \rangle}(v)$ the set of all descendants of $v$ at distance $n$. Iterating we have
\begin{equation*}
    \sum_{u\in\hijos^{\langle n \rangle}(v_0)}|a_u|^2\ge \left(\frac{|\lambda|^2}{2}\right)^n |a_{v_0}|^2.
\end{equation*}
To ensure that $a_v$ is in $\ell^2$ we see that we must have $|\lambda|<\sqrt{2}$.

In fact, one can give a complete spectral description in this case. It is not difficult to check, using similar ideas, that the point spectrum of $S^*$ is empty. Further, by \cite{Jabetal}*{Corollary 3.3.2} the spectra of $S$ and $S^*$ are circularly symmetric, which means that the continuous spectrum consists of all of $\{\lambda: |\lambda|=\sqrt{2}\}$.

\subsection{The (pruned) Collatz graph}\label{sec: Collatz exa}
The Collatz graph is given by choosing $V=\N$ and 
\begin{equation*}
    \Col(n)= \begin{cases}
    n/2  \qquad & \text{if } n \text{ is even,} \\
    (3n+1)/2 \qquad  & \text{if } n \text{ is odd.}\end{cases}
\end{equation*}
The famous Collatz conjecture is equivalent to stating that the Collatz graph is connected. One nice contribution in \cite{Mori}*{Section 4} is to reinterpret this statement in the language of $C^*$-algebras. It would be interesting to find an operator-theoretic formulation of the Collatz conjecture in terms of the shift $S$.

As we mentioned previously, a priori our results cannot be applied to the Collatz graph, because of the presence of cycles. However, after a careful examination, our proof actually applies to any graph $\GG$ that satisfies Assumptions \labelcref{Ass: count,Ass: bdd branch,Ass: leafless} and additionally $\GG\setminus C$ satisfies Assumption \labelcref{Ass: lots of branch}; here $C=\{v: \Col^k(v)=v \text{ for some } k\}$ is the set of vertices in some cycle in $\GG$.

For the Collatz graph we have $R_-=1$ because multiples of $3$ exhibit no branching, i.e., if $n \equiv 0 \mod 3$, then $\hijos^{\langle k \rangle}(n)=\{2^kn\}$ for all $k\in \N$. For this reason Lagarias and Weiss studied the pruned graph, obtained by removing all multiples of $3$. They proved that \cite{LW}*{Section 3}
\begin{equation*}
    2^{1/8}\le R_-\le R_+\le 3^{1/4},
\end{equation*}
and further conjectured that for large $m$
\begin{equation*}
    \kappa_m(n)= \left(\frac43\right)^{m(1+o(1))}
\end{equation*}
for every vertex $n$. As evidence they showed \cite{LW}*{Theorem 3.1} that the average value of $\kappa_m(n)$ (averaged over the vertices $n$) is $(4/3)^m$, which yields
\begin{equation*}
    2^{1/8}\le R_-\le \frac2{\sqrt3}\le R_+\le 3^{1/4}.
\end{equation*}
A computer assisted proof by Applegate and Lagarias, \cite{AL}*{Theorem 1.1}, consisting of examining all trees of depth 30, gives
\begin{equation*}
    (1.302053)^m < \kappa_m(n) < (1.358386)^m
\end{equation*}
for every $n$. They made the stronger\footnote{This conjecture is `uniform in $n$' while the previous conjecture is for fixed $n$.} conjecture \cite{AL}*{Conjecture C} that 
\begin{equation*}
    R_-= R_+ = \frac2{\sqrt3}.
\end{equation*}
Computing the true growth rate appears to be a deep and difficult question.

\section{The spectral radius}\label{sec: 3}
We begin by establishing that $S$ is indeed a bounded operator. We denote by $V^\circ$ the set of vertices that are not a root.

\begin{prop}
    $S$ is bounded and $\|S\|\le\sqrt{B}$.
\end{prop}
\begin{proof}
    If $x=\sum_{v\in V} a_v e_v$ then
    \begin{align*}
        \|Sx\|^2&=\bigg\|\sum_{v\in V} a_v e_{\Col(v)}\bigg\|^2=\bigg\|\sum_{v\in V^\circ} a_v e_{\Col(v)}\bigg\|^2=\bigg\|\sum_{v\in V} \left(\sum_{u\in\hijos{v}}a_u\right) e_{v}\bigg\|^2\\
        &=\sum_{v\in V}\left|\sum_{u\in\hijos{v}}a_u\right|^2\le B\sum_{v\in V}\sum_{u\in\hijos{v}}|a_u|^2\le B\|x\|^2.
    \end{align*}
\end{proof}

We will work with the adjoint, which is easy to describe in terms of $\Col$. 
\begin{prop}
    The adjoint of $S$ is given by
    \begin{equation*}
        S^*\bigg(\sum_{v\in V} a_v e_v\bigg)=\sum_{v\in V^\circ} a_{\Col(v)} e_v.
    \end{equation*}
    In particular, $S^*(e_v)=\sum_{u\in\hijos{v}}e_u$.
\end{prop}

\begin{proof}
    We compute
    \begin{align*}
        \left\langle S \bigg(\sum_{v\in V} a_v e_v\bigg), \sum_{u\in V} b_u e_u\right\rangle &= \sum_{u,v\in V} a_v \overline{b_u} \langle Se_v,e_u\rangle\\ 
        &= \sum_{u\in V,v\in V^\circ} a_v \overline{b_u}\langle e_{\Col (v)},e_u\rangle\\
        &= \sum_{u\in V,v\in \hijos u} a_v \overline{b_u}\\
        &= \sum_{u,v\in V} a_v \overline{b_u} \Big\langle e_{v},\sum_{w\in\hijos{u}}e_w\Big\rangle = \left\langle \sum_{v\in V} a_v e_v, S^* \bigg(\sum_{u\in V} b_u e_u\bigg)\right\rangle.
    \end{align*}
\end{proof}

We are now ready to prove Theorem~\ref{thm1}~\ref{thm: part spec rad}.

\begin{proof}[Proof of Theorem~\ref{thm1}~\ref{thm: part spec rad}]
    Notice that
    \begin{equation*}
        S^m(S^*)^m e_v = S^m \sum_{u\in \hijos^{\langle m\rangle}(v)} e_u = \kappa_m(v) e_v.
    \end{equation*}
    We therefore have
    \begin{equation*}
        \|S^m\|^2 = \|S^m (S^*)^m\|=\kappa_m^{(+)}.
    \end{equation*}
    The result follows by Gelfand's formula for the spectral radius.
\end{proof}
\begin{rmk}
    The same proof shows that $\|S\|=\sqrt{\kappa_1^{(+)}}$. In particular, if the bound $B$ is sharp then $\|S\|=\sqrt{B}$.
\end{rmk}

\section{Universality}\label{sec: univ}
In this section we will give the proof of Theorem~\ref{thm1}~\ref{thm: part univ}. As mentioned previously, our proof is a simple matter of applying Caradus' criterion \cite{Car}, which says that an operator is universal provided that its kernel is infinite-dimensional and its range is the whole space. 

\subsection{Infinite-dimensional eigenspaces} Fix $|\lambda|< R_-$. We first establish that the kernel of $S-\lambda  I$ is infinite-dimensional. We begin by constructing an eigenvector $x=x(v_0)=\sum a_v e_v$ that satisfies $Sx=\lambda x$, where the coefficients $a_v$ are specified by the following construction.  Fix $m$ large enough so that $|\lambda|^{2m}< \kappa_m^{(-)} \le R_-^{2m}$, choose $v_0\in V$ such that $\#\hijos(v_0)\ge2$ and choose $u_+\neq u_-\in\hijos(v_0)$. 

We set $a_{v_0}=0$ and $a_{u_{\pm}}=\pm1$. For $1\le j\le m$ and $v\in\hijos^{\langle j\rangle}(u_\pm)$ we set
\begin{equation*}
    a_v=\pm \frac{\lambda^j \kappa_{m-j}(v)}{\kappa_m(u_\pm)}.
\end{equation*}
Now suppose that we have defined $a_v$ for every $v\in \cup_{j=0}^{km}\hijos^{\langle j\rangle}(u_\pm)$, for some integer $k$. Given $u\in\hijos^{\langle km\rangle}(u_\pm)$ and $1\le j\le m$ we set
\begin{equation*}
    a_v=a_u \frac{\lambda^j \kappa_{m-j}(v)}{\kappa_m(u)}
\end{equation*}   
for $v\in\hijos^{\langle j\rangle}(u)$. This inductively defines $a_v$ for every $v$ that is a descendant of either $u_+$ or $u_-$. We set $a_v=0$ for all other $v\in V$.

Notice that, for any $v\in V$ and $j\ge 1$, we have
\begin{equation*}
    \sum_{u\in\hijos(v)} \kappa_j(u)=\kappa_{j+1}(v).
\end{equation*}
This implies that
\begin{equation}\label{eq: evect coeff}
    \sum_{u\in\hijos(v)} a_u=\lambda a_v
\end{equation}
for any $v\in V$. We will also use $\des(v)$ to denote the set of all descendants of $v$. To see that $x$ is an eigenvector, it therefore suffices to check that
\begin{equation*}
    \sum_{v\in V} |a_v|^2= \sum_{v\in \des(u_+)} |a_v|^2 + \sum_{v\in \des(u_-)} |a_v|^2<+\infty.
\end{equation*}

We have
\begin{equation*}
    \sum_{v\in \des(u_\pm)} |a_v|^2=\sum_{k=0}^\infty \sum_{j=0}^{m-1}\sum_{v\in\hijos^{\langle km+j\rangle}(u_\pm)}|a_v|^2.
\end{equation*}
By construction, the coefficients $a_u,a_v$ have the same argument when $u,v\in\hijos^{\langle t\rangle}(u_+)$ (respectively when $u,v\in\hijos^{\langle t\rangle}(u_-)$). Using now \eqref{eq: evect coeff}, this implies that 
\begin{equation*}
    \sum_{u\in \hijos (v)} |a_u|^2 \le \Big|\sum_{u\in \hijos (v)} a_u\Big|^2 = |\lambda|^2|a_v|^2
\end{equation*}
for $v\in \des(u_\pm)$. Iterating this estimate we get
\begin{equation*}
    \sum_{v\in \des(u_\pm)} |a_v|^2\le \left(\sum_{j=0}^{m-1}|\lambda|^{2j}\right) \sum_{k=0}^\infty \sum_{v\in\hijos^{\langle km\rangle}(u_\pm)}|a_v|^2.
\end{equation*}

Now if $v\in \hijos^{\langle km\rangle}(u_\pm)$ then there exists $u\in \hijos^{\langle (k-1)m\rangle}(u_\pm)$ such that $v\in \hijos^{\langle m\rangle}(u)$. By construction we have
\begin{equation*}
    a_v=a_u \frac{\lambda^m}{\kappa_m(u)}
\end{equation*}
and so
\begin{align*}
    \sum_{v\in\hijos^{\langle km\rangle}(u_\pm)}|a_v|^2 
    &=\sum_{u\in\hijos^{\langle (k-1) m\rangle} (u_\pm)}\sum_{v\in\hijos^{\langle m\rangle}(u)}|a_v|^2\\
    &=\sum_{u\in\hijos^{\langle (k-1) m\rangle} (u_\pm)} \frac{|\lambda|^{2m}}{\kappa_m(u)} |a_u|^2\\ 
    &\le\frac{|\lambda|^{2m}}{\kappa_m^{(-)}}\sum_{u\in\hijos^{\langle (k-1) m\rangle} (u_\pm)}|a_u|^2
\end{align*}
This yields
\begin{equation*}
    \sum_{v\in\hijos^{\langle km\rangle}(u_\pm)}|a_v|^2 \le \left(\frac{|\lambda|^{2m}}{\kappa_m^{(-)}}\right)^k |a_{u_\pm}|^2 = \left(\frac{|\lambda|^{2m}}{\kappa_m^{(-)}}\right)^k
\end{equation*}
and so
\begin{equation*}
    \sum_{v\in \des(u_\pm)} |a_v|^2\le \left(\sum_{j=0}^{m-1}|\lambda|^{2j}\right) \sum_{k=0}^\infty \left(\frac{|\lambda|^{2m}}{\kappa_m^{(-)}}\right)^k < +\infty
\end{equation*}
since $|\lambda|^{2m}<\kappa_m^{(-)}$. We have shown that $x(v_0)$ is an eigenvector.

Finally note that if $v_0$ and $v_1$ are on different branches of the tree, then $x(v_0)$ and $x(v_1)$ are orthogonal. But Assumption~\ref{Ass: lots of branch} 
ensures that there are infinitely many disjoint branches that contain candidates for $v_0$. It follows that the kernel of $S-\lambda  I$ is infinite-dimensional.

\subsection{Surjectivity} We now turn to proving the surjectivity of $S-\lambda  I$. Suppose that $x=\sum a_v e_v$. We proceed to construct vectors $y_j$ such that $Sy_1=x$, $Sy_j=\lambda y_{j-1}$ for $j\ge 2$ and $\sum_j \|y_j\|<+\infty$. Putting $y=\sum_j y_j$ yields $Sy-\lambda y=x$, as desired.

Fix $v\in V$ and $1\le j\le m$ and set
\begin{equation*}
    b_u(v) = a_v \frac{\lambda^{j-1} \kappa_{m-j}(u)}{\kappa_m(v)}.
\end{equation*}
for $u\in\hijos^{\langle j\rangle}(v)$. Now, for a given $k\ge 1$, $1\le j\le m$, $u\in\hijos^{\langle km\rangle}(v)$ and we inductively define
\begin{equation*}
    b_{u'}(v) = b_u(v) \frac{\lambda^j \kappa_{m-j}(u')} {\kappa_m(u)}
\end{equation*}   
for $u'\in\hijos^{\langle j\rangle}(u)$. Finally we put
\begin{equation*}
    y_j(v) = \sum_{u\in\hijos^{\langle j\rangle}(v)} b_u(v) e_u\qquad\text{and}\qquad y_j = \sum_{v\in V} y_j(v),
\end{equation*}
assuming that the sums defining $y_j$ converge.

Notice that
\begin{equation}\label{eq: b coeff ident1}
    \sum_{u\in\hijos(v)} b_u(v)=a_v
\end{equation}
and that
\begin{equation}\label{eq: b coeff ident2}
    \sum_{u'\in\hijos(u)} b_{u'}(v)=\lambda b_u(v)
\end{equation}
for $u\in \hijos^{\langle j\rangle}(v)$ where $j\ge1$. This implies that we have $Sy_1=x$ and $Sy_j=\lambda y_{j-1}$, provided that they are well-defined.

To finish we will show that
\begin{equation*}
    \|y_j\| \le |\lambda|^{j-1} \|x\|
\end{equation*}
for $1\le j\le m$ and 
\begin{equation*}
    \|y_{km+j}\| \le |\lambda|^{j +m-1} \left( \frac{|\lambda|^{m}} {\sqrt{ \kappa_m^{(-)}} }\right)^{k-1} \|x\|
\end{equation*}
for $k\ge 1$ and $1\le j\le m$. Notice that if $v\neq v'$ then $y_j(v)$ and $y_j(v')$ are orthogonal. It follows that
\begin{equation*}
    \|y_j\|^2 = \sum_{v\in V} \sum_{u\in\hijos^{\langle j\rangle}(v)} |b_u(v)|^2.
\end{equation*}
Since the arguments of the coefficients are constant, applying \eqref{eq: b coeff ident1} we get
\begin{equation*}
    \|y_1\|^2 \le \sum_{v\in V} \left|\sum_{u\in\hijos(v)} b_u(v)\right|^2=\sum_{v\in V}|a_v|^2=\|x\|^2.
\end{equation*}
For $j\ge 2$, we apply \eqref{eq: b coeff ident2} instead, yielding
\begin{align*}
    \|y_j\|^2 &=\sum_{v\in V} \sum_{u\in\hijos^{\langle j-1\rangle}(v)} \sum_{u'\in\hijos(u)} |b_{u'}(v)|^2\\
    &\le \sum_{v\in V} \sum_{u\in\hijos^{\langle j-1\rangle}(v)}\left|\sum_{u'\in\hijos(u)} b_{u'}(v) \right|^2\\
    &=|\lambda|^2\sum_{v\in V} \sum_{u\in\hijos^{\langle j-1\rangle}(v)}|b_{u}(v)|^2=|\lambda|^2\|y_{j-1}\|^2.
\end{align*}
It therefore suffices to see that
\begin{equation*}
    \|y_{km}\|^2 \le  \frac{|\lambda|^{2m}} {\kappa_m^{(-)} } \|y_{(k-1)m}\|^2
\end{equation*}
for $k\ge 2$. We have
\begin{align*}
    \|y_{km}\|^2 &=\sum_{v\in V} \sum_{u\in\hijos^{\langle (k-1)m\rangle}(v)} \sum_{u'\in\hijos^{\langle m\rangle}(u)} |b_{u'}(v)|^2\\
    &= \sum_{v\in V} \sum_{u\in\hijos^{\langle (k-1)m\rangle}(v)} \sum_{u'\in\hijos^{\langle m\rangle}(u)} \left|b_{u}(v)\frac{\lambda^m}{\kappa_m(u)}\right|^2\\
    &=|\lambda|^{2m}\sum_{v\in V} \sum_{u\in\hijos^{\langle (k-1)m\rangle}(v)}\frac{|b_{u}(v)|^2}{\kappa_m(u)}\\
    &\le \frac{|\lambda|^{2m}} {\kappa_m^{(-)}} \sum_{v\in V} \sum_{u\in\hijos^{\langle (k-1)m\rangle}(v)} |b_{u}(v)|^2= \frac{|\lambda|^{2m}} {\kappa_m^{(-)}} \|y_{(k-1)m}\|^2.
\end{align*}
This completes the proof.

\subsection*{Declaration on AI usage}
AI was used only to improve the presentation, but all of the mathematical arguments presented here are our own.

\end{document}